\documentclass[11pt]{amsart}

\usepackage{amsmath,amsfonts,amsthm,amsopn,amssymb}
\usepackage{verbatim,wasysym,mathrsfs}
\usepackage[left=2.6cm,right=2.6cm,top=3.1cm,bottom=3.1cm]{geometry}
\usepackage{cite,marginnote}
\usepackage{color,enumitem,graphicx}
\usepackage[colorlinks=true,urlcolor=blue, citecolor=red,linkcolor=blue,
linktocpage,pdfpagelabels, bookmarksnumbered,bookmarksopen]{hyperref}
\usepackage[english]{babel}
\usepackage[T1]{fontenc}
\usepackage{lmodern}
\usepackage[hyperpageref]{backref}

\numberwithin{equation}{section}

\makeindex
\newtheorem{theorem}{Theorem}[section]
\newtheorem{proposition}[theorem]{Proposition}
\newtheorem{lemma}[theorem]{Lemma}
\newtheorem{remark}[theorem]{Remark}
\newtheorem{example}[theorem]{Example}
\newtheorem{corollary}[theorem]{Corollary}
\newtheorem{definition}[theorem]{Definition}

\newcommand{\s}{\section}

\newcommand{\R}{\mathbb R}

\newcommand{\bt}{\begin{theorem}}
\newcommand{\et}{\end{theorem}}
\newcommand{\bl}{\begin{lemma}}
\newcommand{\el}{\end{lemma}}
\newcommand{\bd}{\begin{definition}}
\newcommand{\ed}{\end{definition}}
\newcommand{\bc}{\begin{corollary}}
\newcommand{\ec}{\end{corollary}}
\newcommand{\bp}{\begin{proof}}
\newcommand{\ep}{\end{proof}}
\newcommand{\bx}{\begin{example}}
\newcommand{\ex}{\end{example}}
\newcommand{\bi}{\begin{exercise}}
\newcommand{\ei}{\end{exercise}}
\newcommand{\bo}{\begin{proposition}}
\newcommand{\eo}{\end{proposition}}
\newcommand{\br}{\begin{remark}}
\newcommand{\er}{\end{remark}}
\newcommand{\be}{\begin{equation}}
\newcommand{\ee}{\end{equation}}
\newcommand{\ba}{\begin{align}}
\newcommand{\ea}{\end{align}}
\newcommand{\bn}{\begin{enumerate}}
\newcommand{\en}{\end{enumerate}}
\newcommand{\bg}{\begin{align*}}
\newcommand{\bcs}{\begin{cases}}
\newcommand{\ecs}{\end{cases}}

\def\R{\mathbb R}

\makeatletter
\def\@makefnmark{}
\makeatother

\newcommand{\bean}{\begin{eqnarray*}}
\newcommand{\eean}{\end{eqnarray*}}

\renewcommand{\triangle}{\Delta}
\renewcommand{\epsilon}{\varepsilon}

\title[Existence and multiplicity of solutions for  Schr\"{o}dinger-Born-infeld system]
{A perturbation approach \\for the Schr\"{o}dinger-Born-infeld system:\\
 solutions
in the subcritical and critical case}

\author[Z. S. Liu]{Zhisu Liu}
\author[G. Siciliano]{Gaetano Siciliano}

\address[Z. S. Liu]{\newline\indent School of Mathematics and Physics, China University of Geosciences,
\newline\indent
Wuhan, Hubei, 430074, PR China}
\email{\href{mailto:liuzhisu183@sina.com}{liuzhisu183@sina.com}}

\address[G. Siciliano]{\newline\indent
Departamento de Matem\'atica
\newline\indent
Instituto de Matem\'atica e Estat\'istica
\newline\indent
 Universidade de S\~ao Paulo
\newline\indent
Rua do Mat\~ao 1010,  05508-090 S\~ao Paulo, SP, Brazil }
\email{\href{mailto:sicilian@ime.usp.br}{sicilian@ime.usp.br}}

\thanks{Z. Liu is supported by the NSFC (11626127);
and Hunan Natural Science Excellent Youth Fund (2020JJ3029).
G. Siciliano was supported by Fapesp grant 2018/17264-4, CNPq grant 304660/2018-3 and Capes
(Brazil) and INdAM (Italy).}


\begin{document}

\begin{abstract}
In this paper, we study the following Schr\"{o}dinger-Born-infeld system with a general nonlinearity
$$
\left\{
 \begin{array}{ll}
-\triangle u+u+\phi u=f(u)+\mu|u|^4u\,\,&\mbox{in}\,\,\R^3,\\
-\textrm{div}\displaystyle\bigg(\frac{\nabla\phi}{\sqrt{1-|\nabla\phi|^2}}\bigg)=u^2&\mbox{in}\,\,\R^3,\\
u(x)\rightarrow0,\,\,\phi(x)\rightarrow0,&\,\text{as}\,\,x\rightarrow\infty,
\end{array}
\right.
$$
where $\mu\geq0$ and $f\in C(\R,\R)$ satisfies suitable assumptions.
 This system arises from a suitable coupling of the nonlinear Schr\"{o}dinger
equation and the Born-Infeld theory.
We use a new perturbation approach to prove the existence and multiplicity of nontrivial solutions of the above system
in the  subcritical and  critical case.
We emphasise that our results  cover the case $f(u)=|u|^{p-1}u$ for $p\in(2,{5}/{2}]$ and $\mu=0$
which was left  in \cite{Azzollini19} as an open problem.
\end{abstract}
\maketitle
\begin{center}
\begin{minipage}{12cm}
\tableofcontents
\end{minipage}
\end{center}

\section{Introduction}
\label{sec1}

\subsection{Overview}
We are concerned with the following Schr\"{o}dinger-Born-Infeld system
\begin{equation}\tag{SBI}\label{SBI}
\left\{
 \begin{array}{ll}
-\triangle u+u+\phi u=f(u)+\mu|u|^{4}u\,\,&\mbox{in}\,\,\R^3,\medskip\\
-\displaystyle\text{div}\bigg(\frac{\nabla\phi}{\sqrt{1-|\nabla\phi|^2}}\bigg)=u^2&\mbox{in}\,\,\R^3,\medskip\\
u(x)\rightarrow0,\,\,\phi(x)\rightarrow0,&\,\,\text{as}\,\,x\rightarrow\infty,
\end{array}
\right.
\end{equation}
where  $f\in C(\R,\R)$ is a suitable nonlinearity and $\mu\geq0$.
Evidently when $\mu>0$ we are in the critical case.

Such a type of system arises from the coupling of the nonlinear Schr\"{o}dinger
equations with the equations of the electromagnetic field in  the Born-Infeld theory, and can be proposed to provide a mathematical description of the
interaction between a charged particle and the electromagnetic field generated by itself.
The Born-Infeld theory was firstly developed by Born and Infeld and introduced the idea that
both the matter and the electromagnetic field were expression of a unique physical entity.

Unlike systems involving  the  Maxwell lagrangian for the electromagnetic field
(such as the coupling with the Schr\"odinger or Klein-Gordon equation), systems which involve the
Born-Infeld lagrangian are  much less studied in the mathematical literature
especially due to the difficulties related to the mean curvature operator appearing in the
second equation (the electrostatic field equation).

We cite here the papers \cite{BK, GG1,GG2, IKL} where a second order approximation of the
Born-Infeld lagrangian is used and the second equation
is then reduced to the quasilinear equation
$$-\Delta\phi -\Delta_{4}\phi = u^{2}.$$
In particular in \cite{GG1,GG2} the existence of solutions
is studied by variational methods mixed with truncation techniques.

To the best of our knowledge the first paper which deals with
the mean curvature operator in such kind of systems is the one by
 Yu \cite{Yu2010}. Here the author considers the  coupling
of the Klein-Gordon with the Born-Infeld lagrangian and in the electrostatic case 
the search of standing waves solutions reduces to study the system
\begin{equation}\label{eq:Yu}
\left\{
 \begin{array}{ll}
-\triangle u+(m^2-(\omega+\phi)^2)u=|u|^{p-1}u\,\,&\mbox{in}\,\,\R^3,\medskip\\
-\text{div}\displaystyle\bigg(\frac{\nabla\phi}{\sqrt{1-|\nabla\phi|^2}}\bigg)=u^2(\omega+\phi)&\mbox{in}\,\,\R^3, \medskip\\
u(x)\rightarrow0,\,\,\phi(x)\rightarrow0,&\,\,\text{as}\,\,x\rightarrow\infty.
\end{array}
\right.
\end{equation}
Evidently there are some difficulties related to the operator appearing
in the second equation.
One  can not use variational approach to deal with this problem by restricting the functional at the usual
function spaces. The reason is that the quantity $1/\sqrt{1-|\nabla \phi(x)|^2}$ makes sense only when
$x\in\R^3$ is such that $|\nabla \phi(x)|<1$.
And so this inequality has to be considered in the functional setting as a necessary constraint.

We cite also the paper \cite{BdP} where the authors consider
the electrostatic Born-Infeld equation, the second equation in \eqref{eq:Yu},
with extended charges as sources.

Inspired by  \cite{Yu2010}, Azzollini, Pomponio and Siciliano in \cite{Azzollini19} proposed
and introduced a new model which represents a variant of the well-known Schr\"{o}dinger-Maxwell
system as it was introduced in Benci and Fortunato \cite{Benci98}, see also
and D'Aprile and Mugnai \cite{DM}. By replacing the usual Maxwell lagrangian with
the Born-Infeld one, the authors
in \cite{Azzollini19} studied the existence of electrostatic solutions
which lead to the following system
\begin{equation}\label{eqn:SBI}
\left\{
 \begin{array}{ll}
-\triangle u+u+\phi u=|u|^{p-1}u\,\,&\mbox{in}\,\,\R^3,\medskip\\
-\textrm{div}\displaystyle\bigg(\frac{\nabla\phi}{\sqrt{1-|\nabla\phi|^2}}\bigg)=u^2&\mbox{in}\,\,\R^3,\medskip\\
u(x)\rightarrow0,\,\,\phi(x)\rightarrow0,&\,\,\text{as}\,\,x\rightarrow\infty.
\end{array}
\right.
\end{equation}
In particular they use a slightly modified version of the monotonicity trick due to
Jeanjean \cite{Jeanjean99} and Struwe \cite{Struwe85} to
prove the existence of radial ground state solutions
of (\ref{eqn:SBI}) for $p\in(5/2, 5)$. They left as an open problem the case of smaller $p$ and the existence
of non-radial solutions.

\subsection{Main results}
In this paper one of our aim is to prove the existence of
ground state for  problem \eqref{eqn:SBI} also in the case $p\in (2,5/2]$.
Actually we will consider a more general nonlinearity than the power one.
Moreover we establish  a multiplicity results and treat also the
problem with a critical nonlinearity.
However we are not able to avoid the radial setting.

\medskip

We recall that \eqref{SBI} comes variationally from the action functional $F$ defined by
$$
F(u,\phi)=\frac{1}{2}\int_{\R^3}(|\nabla u|^2+u^2)+\frac{1}{2}\int_{\R^3}\phi u^2
-\frac{1}{2}\int_{\R^3}(1-\sqrt{1-|\nabla \phi|^2})-\int_{\R^3}F(u)-\frac{\mu}{6}\int_{\R^3}|u|
^6.
$$
which presents evident difficulties.
As introduced in \cite{Yu2010},
the presence of the term $\int_{\R^3}(1-\sqrt{1-|\nabla \phi|^2})$ forces us to restrict the setting of
admissible function $\phi$ and to define a suitable function set.
Define
$$
X:=D^{1,2}(\R^3)\cap\{\phi\in C^{0,1}(\R^3):\,\|\nabla \phi\|_{\infty}\leq1\}
$$
where $D^{1,2}(\R^3)$ is the completion of $C_0^\infty(\R^3)$ with respect to the norm $\|\nabla\cdot\|_2$.
Let as usual
$$
H_r^1(\R^3)=\{u\in H^1(\R^3):  u \text{ is radially symmetric}\}
$$
and
$$
X_r=\{\phi\in X\,\big|\,\phi\text{ is radially symmetric}\}.
$$

We will study the functional $F$ by restricting on the function setting $H^1_{r}(\R^3)\times X_{r}$,
which is not a vector space.
This brings us some difficulties for the variational approach;
indeed to compute variations with respect
to $\phi$ along the direction established by a generic smooth and compactly supported
function, a direct restriction is to require in advance that $\|\nabla \phi\|_\infty<1$.
This results in a nontrivial obstacle to find the relation between solutions of the minimizing problem
and solutions of the second equation of system \eqref{SBI} for $u$ fixed, see Yu \cite{Yu2010}.
In \cite{BdP} the author obtained the uniqueness of the solution of the second equation
 in the radial setting which will be useful in our analysis.

Moreover, the functional $F$ is strongly indefinite, but this is usually overcome
by the classical  {\sl reduction method} with the help of the results obtained in \cite{BdP}.
Precisely,
for any radial $u\in H_r^1(\R^3)$ fixed, there exists
a unique $\phi_u\in X_{r}$ solution of the second equation of system \eqref{SBI} 
and this reduces
the problem to that of finding critical points of the (no more strongly indefinite)
one variable functional $I(u)=F(u, \phi_u)$ defined on $H^1_{r}(\R^3)$.
It is known that the radial setting is a natural constraint for the problem,
see \cite[Proposition 2.6]{Azzollini19}.

\medskip

In order to state our results let us state the assumptions on the nonlinearity $f$:
\begin{enumerate}[label=(f\arabic*),ref=f\arabic*,start=1]
\item
\label{f1} $f\in C(\R,\R)$ and $\lim_{s\to 0}f(s)/s=0$; \medskip
\item
 \label{f2} $|f(s)|\leq C(1+|s|^{p})$ for $p\in(2,5)$; \medskip
\item
 \label{f3} for any $s>0$, $0<\varrho F(s)\leq f(s)s$, where $\varrho\in(3,4)$ and $F(s)=\int_{0}^sf(\tau)d\tau$.
\end{enumerate}

These are quite natural assumptions when dealing with variational methods.
In particular by \eqref{f1}-\eqref{f2} it follows that
 for any $\epsilon>0$,
there exists $C_{\epsilon} > 0$ such that
\begin{equation}\label{eq:fF}
|f(s)|\leq \epsilon |s|+C_{\epsilon}|s|^{p} \quad\text{and}\quad |F(s)| \leq \varepsilon s^{2}+C_{\varepsilon}|s|^{p+1}.
\end{equation}
On the other hand  \eqref{f3} is useful when dealing with the Mountain Pass
geometry and Palais-Smale sequences.
From \eqref{f2} and \eqref{f3} it follows that for suitable constants
$$C_{1}|s|^{\varrho} - C_{2}\leq C_{3} s^{2} + C_{4}|s|^{p+1},$$
which allowes  $p>2$.

The following is our first result in this paper and concerns  the subcritical case.
\begin{theorem}\label{Thm:EXST}
Assume that \eqref{f1}-\eqref{f3} hold. Then system \eqref{SBI}
with $\mu=0$
 has at least a radial ground state solution
 , namely, a solution
 $(u,\phi)\in H_r^1(\R^3)\times X_r$ minimizing the energy functional $F$
 among all the nontrivial radial solutions.
\end{theorem}

In particular
if $f$ is a pure power nonlinearity
this result improves the one in \cite{Azzollini19} so
extending the existence of ground state to all the range
$p\in (2,5)$.

A second result concerns multiplicity of solutions again in the subcritical case.
\begin{theorem}\label{Thm:multi}
Assume  that \eqref{f1}-\eqref{f3}  hold and $f$ is odd. Then system  \eqref{SBI}
with $\mu=0$
 has infinite many radial high energy solutions, namely, solutions
 $(u_j,\phi_j)\subset H_r^1(\R^3)\times X_r$ such that the energy functional
$F$ tends to infinity.
\end{theorem}

In what follows, we turn our attention to the case where the nonlinearity has critical growth
and without loss of generality we assume $\mu=1$.
For this purpose, and to deal with compactness issues, we consider the further assumption on the nonlinearity $f$:
\begin{enumerate}[label=(f\arabic*),ref=f\arabic*,start=4]
\item 
\label{f4}  there exist $D>0$ and $2<r<6$ such that $f(t)\geq Dt^r$
for $t\geq0$.
\end{enumerate}
\begin{theorem}\label{th:crit}
Assume that \eqref{f1}-\eqref{f4} hold.
Then system \eqref{SBI} with $\mu=1$
 admits a radial ground state solution $(u,\phi)\in H_r^1(\R^3)\times X_r$
 in the following cases
 \begin{itemize}
 \item[1)] $r\in(4, 6)$, or \smallskip
 \item[2)] $r\in(2, 4]$ and  $D$ is sufficiently large.
 \end{itemize}
\end{theorem}

We remark that, arguing as in \cite{Azzollini19}, all the solutions found
are of class $C^2(\R^3)$,
hence classical.

\medskip

As we said before,
to prove our results we use variational methods.
We point out that the monotonicity trick used in \cite{Azzollini19}
does not permit to conclude the existence of infinitely many solutions
 for the case $p\leq 5/2$, so a new
perturbation technique inspired by \cite{Liu19,Liu19-,Liu19--} is introduced here in order to deal with this case.
We point out that no truncation technique is used here.

This is our main contribution and indeed
we think that this type of perturbations  can be used also for other problems.
Unfortunately we are not able to treat the case $p\in(1,2]$ in \eqref{eqn:SBI}
that we believe it is interesting too.

\medskip

This paper is organized as follows. Firstly, some
 preliminaries are given in Section \ref{sec:prel}. Section \ref{sec:subcritical} is devoted to the subcritical case:
after introducing the modified perturbed equation,
 the existence of ground state solutions and multiplicity of high energy solutions for \eqref{SBI}
 with a general subcritical nonlinearity are shown, then proving Theorem \ref{Thm:EXST} and Theorem
 \ref{Thm:multi}.
 In the final Section \ref{sec:critical} we treat the critical case: again after
 introducing a perturbation in the equation, we prove the existence of a radial ground state solution of
\eqref{SBI} with a general critical nonlinearity.

\subsection*{Notations}
As a matter of notations, we will use $C,C_{1},C_{2},\ldots,C', \ldots$ to denote suitable positive constants
whose value may vary from line to line.

For every $1\leq s\leq +\infty$, we denote by $\|\cdot\|_s$ the usual
norm of the Lebesgue space $L^s(\R^3)$ and use $\|\cdot\|$
for the standard norm in $H^{1}_{r}(\mathbb R^{3})$.
We will use the ``small o'' notation for  vanishing sequences.
Other notations will be introduced whenever we need.

\s{Preliminary results}\label{sec:prel}

We recall some properties of the ambient space $X$.
For the proofs see \cite{BdP} or \cite[Lemma 2.1 and Remark 2.3]{Azzollini19}.
\begin{lemma}\label{Lem:compact0}
The following conclusions hold:
\begin{itemize}
	
	\item [$(i)$] $X$ is continuously embedded in $W^{1,p}(\R^3)$ for all $p\in[6,+\infty)$; \smallskip
	
	\item [$(ii)$]   $X$ is continuously embedded in $L^{\infty}(\R^3)$;\smallskip
	
	\item [$(iii)$]  if $\phi\in X$, then $\lim_{|x|\rightarrow\infty}\phi(x)=0$;\smallskip

	\item [$(iv)$] $X$ is weakly closed;\smallskip
	\item [$(v)$] if $\{\phi_n\}\subset X$ is  bounded, there exists $\phi\in X$ such that, up to subsequence,
$\phi_n\rightharpoonup\phi$ weakly in $X$ and uniformly in compact sets in $\mathbb R^{3}$;\smallskip
\item[$(vi)$] if $u_{n}\to u$ in $L^{p}(\mathbb R^{3}), p\in[1,+\infty)$, then $\phi_{u_{n}}\to \phi_{u}$ in $L^{\infty}(\mathbb R^{3})$.
\end{itemize}
\end{lemma}

We define weak solutions to \eqref{SBI} in the following way.
\begin{definition}\label{Def:compact0}
A weak solution of \eqref{SBI} is a couple $(u,\phi)\in H^1(\R^3)\times X$ such that for all
$(v,\psi)\in C_0^\infty(\R^3)\times C_0^\infty(\R^3)$, we have
$$
 \begin{array}{ll}
\displaystyle \int_{\R^3}(\nabla u\nabla v+uv)+\int_{\R^3}\phi_uuv=\int_{\R^3}f(u)v+\mu\int_{\R^3}|u|^4u\psi, \medskip\\
\displaystyle \int_{\R^3}\frac{\nabla\phi\nabla\psi}{\sqrt{1-|\nabla\phi|^2}}=\int_{\R^3}u^2\psi.
\end{array}
$$
We can even allow here $v,\psi\in H^{1}(\mathbb R^{3})$.
\end{definition}

The next facts are also known, see \cite[Lemma 2.2]{Azzollini19}.

\begin{lemma}\label{Lem:phi}
For any $u\in H^1(\R^3)$ fixed, there exists a unique $\phi_u\in X$ such that the following properties hold:
\begin{enumerate}[label=(\roman*),ref=\roman*]
\item \label{Lem:phi-i}  $\phi_u$ is the unique minimizer of $E_u:X\rightarrow\R$ defined as
$$
E_u(\phi)=\int_{\R^3}(1-\sqrt{1-|\nabla\phi|^2})-\int_{\R^3}\phi u^2,
$$
 and $E_u(\phi_u)\leq 0$, that is,
	$$
\int_{\R^3}\phi_uu^2\geq \int_{\R^3}(1-\sqrt{1-|\nabla\phi_u|^2});
    $$
	\item \label{Lem:phi-ii}
	  $\phi_u\geq0$ and $\phi_u=0$ if and only if $u=0$;
	
	\item  \label{Lem:phi-iii} if $\phi$ is a weak solution of the second equation of \eqref{SBI}, then $\phi=\phi_u$ and satisfies the following
equality
$$
\int_{\R^3} \frac{|\nabla \phi_u|^2}{\sqrt{1-|\nabla \phi_u|^2}}=\int_{\R^3}\phi_uu^2.
$$
Moreover, if $u\in H_r^1(\R^3)$, then $\phi_u\in X_r$ is the unique weak solution of the second equation of system
\eqref{SBI}.
\end{enumerate}
\end{lemma}

Actually the inequality in \eqref{Lem:phi-i} can be improved. Indeed,
since for all $t\in[0,1)$, the following holds:
\begin{equation}\label{eqn:PS2}
1-\sqrt{1-t}\leq\frac{1}{2}\frac{t}{\sqrt{1-t}},
\end{equation}
by recalling \eqref{Lem:phi-iii}  in Lemma \ref{Lem:phi}, it follows
\begin{equation}\label{eq:1/2}
\frac{1}{2}\int_{\mathbb R^{3}} \phi_{u}u^{2} = \frac{1}{2}\int_{\R^3} \frac{|\nabla \phi_u|^2}{\sqrt{1-|\nabla \phi_u|^2}}
\geq  \int_{\R^3}(1-\sqrt{1-|\nabla\phi_u|^2}).
\end{equation}

We recall also a Pohozaev type identity associated with \eqref{SBI} whose proof
can be obtained as in \cite{Azzollini19}.
\begin{lemma}\label{Lem:poh}
If $(u,\phi)\in H^1_r(\R^3)\times X_r$ is a solution of \eqref{SBI}, then the following Pohozaev type identity is satisfied:
\begin{equation*}
\frac{1}{2}\int_{\R^3}|\nabla u|^2+\frac{3}{2}\int_{\R^3}u^2+2\int_{\R^3}\phi u^2
-\frac{3}{2}\int_{\R^3}(1-\sqrt{1-|\nabla \phi|^2})=3\int_{\R^3}F(u)+\frac{\mu}{2}\int_{\R^3}|u|^6.
\end{equation*}
\end{lemma}

In view of Lemma \ref{Lem:phi}, we know that the associated energy functional can be written
in the following form
\begin{equation}\label{eq:I}
I(u)=\frac{1}{2}\int_{\R^3}(|\nabla u|^2+u^2)+\frac{1}{2}\int_{\R^3}\phi_uu^2
-\int_{\R^3}F(u)-\frac{1}{2}\int_{\R^3}(1-\sqrt{1-|\nabla \phi_u|^2})-\frac{\mu}{6}\int_{\R^3}|u|^6,
\end{equation}
and is of class $C^1$. It is known that if $u$ is the critical point of $I$,
then $(u,\phi_u)$ is a weak solution of \eqref{SBI}. For this reason
we will speak simply of  $u$ as a  solution of \eqref{SBI}.

Finally we recall a technical lemma which is of use in studying the geometry of the functional,
see \cite[Lemma 2.7]{Azzollini19}.
\begin{lemma}\label{Lem:tech}
Let $s\in [2,3)$. Then there exist positive constants $C$ and $C'$ such that for any $u\in H^1(\R^3)$, we have
$$
\|\nabla \phi_u\|_2^{\frac{s-1}{s}}\leq C\|u\|_{2(s^*)'}\leq C'\|u\|,
$$
where $s^*$ is the critical Sobolev exponent related to $s$ and $(s^*)'$ is its conjugate exponent, namely
$$
s^*=\frac{3s}{3-s}\quad\text{and}\quad(s^*)'=\frac{3s}{4s-3}.
$$
In particular,
$$\int_{\mathbb R^{3}} \phi_{u} u ^{2}\leq \|\phi_{u}\|_{6}\| u\|_{12/5}^{2}\leq
\|\nabla \phi_{u}\|_{6}\| u\|_{12/5}^{2}\leq C' \|u\|_{12/5}^{4}.$$
\end{lemma}

\section{The Subcritical Case}\label{sec:subcritical}
\setcounter{equation}{0}
\subsection{The modified equation}
In this section, we are planning to study \eqref{SBI} in the subcritical case, that is when
$\mu=0$. Here,
 we introduce \emph{a perturbation technique} to overcome
 the difficulty
related to the  boundedness of Palais-Smale sequence.
 \eqref{SBI}.
We consider the following modified problem
\begin{equation}\label{eqn:2-spos}
\left\{
 \begin{array}{ll}
-\triangle u+u+\phi u+\lambda\|u\|_2u=f(u)+\lambda |u|^{q-1}u\,\,&\mbox{in}\,\,\R^3,\medskip\\
-\text{div}\bigg(\displaystyle\frac{\nabla\phi}{\sqrt{1-|\nabla\phi|^2}}\bigg)=u^2&\mbox{in}\,\,\R^3,\medskip\\
u(x)\rightarrow0,\,\,\phi(x)\rightarrow0,&\,\text{as}\,\,x\rightarrow\infty,
\end{array}
\right.
\end{equation}
where $\lambda\in(0,1]$ 
and $q\in(\max\{p,4\},5)$. Thus,
its associated functional is
$$
I_{\lambda}(u)=I(u)+\frac{\lambda}{3}\|u\|_2^{3}
-\frac{\lambda}{q+1}\int_{\R^3}|u|^{q+1},
$$
where $I$ is defined in \eqref{eq:I}, with $\mu=0$.

The idea is to find solutions for \eqref{eqn:2-spos} and then send $\lambda$ to $0$.

We now verify that the functional $I_\lambda$ has  the Mountain Pass geometry
uniformly in $\lambda$.

\begin{lemma}\label{Lem:MP1}
Suppose that \eqref{f1}-\eqref{f3}  hold.
 Then 
 \begin{enumerate}[label=(\roman{*}), ref=\roman{*}]
 \item\label{MPi} there exist $\rho,\delta>0$ such that, for any $\lambda\in(0,1]$,
  $I_{\lambda}(u)\geq \delta$ for every $u\in S_\rho=\{u\in E: \|u\|=\rho\}$;
 \medskip
 \item\label{MPii} there is $v\in H_r^1(\R^3)$ with $\|v\|>\rho$ such that,
 for any $\lambda\in(0,1]$, $I_{\lambda}(v)<0$.
 \end{enumerate}
\end{lemma}
\begin{proof}
\eqref{MPi}
For any $u\in H_r^1(\R^3)$, by the definition of $I_\lambda$, Lemma \ref{Lem:phi}-\eqref{Lem:phi-i}
and \eqref{eq:fF} involving $q$, one has
\begin{equation*}
\aligned
I_{\lambda}(u)
&\geq\frac{1}{2}\|u\|^2-\int_{\R^3}F(u)-\frac{\lambda}{q+1}\int_{\R^3}|u|^{q+1} \\
&\geq\frac{1}{2}\|u\|^2
-\epsilon\|u\|_2^2
-\frac{1+C_\epsilon}{q+1}\int_{\R^3}|u|^{q+1}\\
&\geq\frac{1-\epsilon}{2}\|u\|^2
-\frac{1+C_\epsilon}{q+1}C\|u\|^{q+1}.
\endaligned
\end{equation*}
Taking  $\epsilon={1}/{2}$, and  $\|u\|=\rho>0$ small
enough, it is easy to check that there exists $\delta>0$ such that
$I_{\lambda}(u)\geq \delta$ for every $u\in S_{\rho}$. \\

\eqref{MPii}
By  the definition of $I_\lambda$ and Lemma \ref{Lem:tech}  it follows
\begin{equation}\label{eqn:MPG}
\aligned
I_{\lambda}(u)
&\leq\frac{1}{2}\|u\|^2
+\frac{1}{2}\int_{\R^3}\phi_uu^2+\frac{\lambda}{3}\|u\|_2^{3}-\int_{\R^3}F(u)
-\frac{\lambda}{q+1}\int_{\R^3}|u|^{q+1}\\
&\leq\frac{1}{2}\|u\|^2
+C_{1}\|u\|_{{12}/{5}}^4+C_{2}\|u\|_2^{3}
-\int_{\R^3}F(u).
\endaligned
\end{equation}
For $e\in H_r^1(\R^3)\setminus\{0\}$, let  $e_{t}=t^2e(tx)$
and observe that
\begin{equation*}\label{eqn:MPG2}
\int_{\R^3}F(e_{t})=t^{-3}\int_{\R^3}F(t^2 e)=:t^{-3} \Phi(t).
\end{equation*}
By \eqref{f3}, a straightforward computation yields
$$
\frac{\Phi'(t)}{\Phi(t)}\geq \frac{2\varrho}{t} \quad \forall t>0
$$
and then, by integrating on $[1, t]$, with $t>1$, we have
$\Phi(t)\geq \Phi(1)t^{2\varrho}$, implying that
\begin{equation}\label{eq:F}
\int_{\mathbb R^{3}} F(e_{t}) \geq t^{2\varrho -3}\int_{\mathbb R^{3}}F(e).
\end{equation}
Consequently by \eqref{eqn:MPG}  and \eqref{eq:F} we infer
\begin{equation*}\label{eqn:MPG1}
I_{\lambda}(e_{t})\leq\frac{t^3}{2}\int_{\R^3}|\nabla e|^2+\frac{t}{2}\int_{\R^3}e^2
+C_{1}t^3\|e\|_{{12}/{5}}^4+C_{2} t^{3/2}
\|e\|_2^{3}
-t^{2\rho-3}\int_{\R^3}F(e).
\end{equation*}
Since by \eqref{f3} it is $\varrho>3$, the conclusion holds with $t$  large enough.
\end{proof}

By the well-known Mountain-Pass theorem (see \cite{AR,Willem96}),
there exists a $(PS)_{c_{\lambda}}$ sequence
$\{u_n\}\subset H_r^1(\R^3)$, that is,
\begin{equation*}
I_{\lambda}(u_n)\rightarrow c_{\lambda}\quad\text{and}\quad I'_{\lambda}(u_n)\rightarrow0.
\end{equation*}
It is clear that $\{u_{n}\}$ depends on $\lambda$ but we omit this dependence in the sequel.
Here $c_{\lambda}$ is the Mountain Pass level characterized by
\begin{equation*}
c_{\lambda}=\inf\limits_{\gamma\in\Gamma_{\lambda}}\max\limits_{t\in[0,1]}I_{\lambda}(\gamma(t))
\end{equation*}
with
$$
\Gamma_{\lambda}:=\left\{\gamma\in C^1([0,1],H_{r}^1(\R^3)):\,\gamma(0)=0\quad\text{and}\quad I_{\lambda}(\gamma(1)) <0\right\}.
$$
\begin{remark}\label{rem:bounded}
Observe from Lemma \ref{Lem:MP1} that there exist two constants $m_{1}, m_{2}>0 $
independently on $\lambda$ such that $m_{1}<c_\lambda<m_{2}$.
\end{remark}

We state the following lemma to ensure that Palais-Smale sequences of $I_{\lambda}$
 at level $c_\lambda$ have at least a convergence subsequence.

\begin{lemma}\label{Lem:compact}
For fixed 
$\lambda\in (0, 1]$, let $\{u_n\}\subset H_r^1(\R^3)$ be
a $(PS)$ sequence of $I_{\lambda}$.
Then there exists $u\in H_r^1(\R^3)$ such that
$u_n\rightarrow u$ in $H_r^1(\R^3)$.
\end{lemma}
\begin{proof}
The proof is divided into two parts.

In the first one we show that the sequence $\{u_n\}$ is bounded in $H_r^1(\R^3)$.
For $\theta\in(4,q+1)$, by Lemma \ref{Lem:tech}, there exist
$C_1,C_2>0$ such that
\begin{eqnarray*}\label{eqn:PS1}
C_1+C_2\|u_n\|&\geq& I_{\lambda}(u_n)-\frac{1}{\theta}I'_\lambda(u_n)u_n\\
&= &\frac{\theta-2}{2\theta}\|u_n\|^2+\frac{\theta-2}{2\theta}\int_{\R^3}\phi_{u_n}u_n^2
+\frac{\theta-3}{3\theta}\lambda\|u_n\|_2^{3}\\
&&+\int_{\R^3}(\frac{1}{\theta}f(u_n)u_n-F(u_n))-\frac{1}{2}\int_{\R^3}(1-\sqrt{1-|\nabla \phi_{u_n}|^2})
+\frac{q+1-\theta}{\theta(q+1)}\lambda\int_{\R^3}|u_n|^{q+1}\\
&\geq &\frac{\theta-2}{2\theta}C_3\|u_n\|^2+\frac{\theta-2}{2\theta}\int_{\R^3}\phi_{u_n}u_n^2
+\frac{\theta-3}{3\theta}\lambda\|u_n\|_2^{3}\\
&&-C_4\int_{\R^3}|u_n|^{p+1}-\frac{1}{2}\int_{\R^3}(1-\sqrt{1-|\nabla \phi_{u_n}|^2})
+\frac{q+1-\theta}{\theta(q+1)}\lambda\int_{\R^3}|u_n|^{q+1}.
\end{eqnarray*}
By recalling \eqref{eq:1/2} it follows that
\begin{eqnarray}\label{eqn:PS3}
\aligned
C_1+C_2\|u_n\|
\geq & \ \frac{\theta-2}{2\theta}C_3\|u_n\|^2+\frac{\theta-4}{4\theta}\int_{\R^3}\phi_{u_n}u_n^2
+\frac{\theta-3}{3\theta}\lambda\|u_n\|_2^{3}\\
&-C_4\int_{\R^3}|u_n|^{p+1}
+\frac{q+1-\theta}{\theta(q+1)}\lambda\int_{\R^3}|u_n|^{q+1}\\
\geq &\ \frac{\theta-2}{2\theta}C_3\|u_n\|^2
+\frac{\theta-3}{3\theta} \lambda \|u_n\|_2^{3}\\
&-C_4\int_{\R^3}|u_n|^{p+1}
+\frac{q+1-\theta}{\theta(q+1)}\lambda\int_{\R^3}|u_n|^{q+1}.
\endaligned
\end{eqnarray}
Observe that for any large $B_1>0$, there exists $B_2>0$ such that
$$
\frac{\theta-3}{3\theta}\|u_n\|_2^{3}\geq B_1\|u_n\|_2^2-B_2,
$$
which, together with (\ref{eqn:PS3}), implies that
\begin{equation}\label{eqn:PS4}
\aligned
C_1+\lambda B_2+C_2\|u_n\|
\geq &\frac{\theta-2}{2\theta}C_3\|u_n\|^2+\int_{\R^3}\bigg(\lambda B_1|u_n|^2
-C_4|u_n|^{p+1}
+\frac{q+1-\theta}{\theta(q+1)}\lambda|u_n|^{q+1}\bigg).
\endaligned
\end{equation}
We note that $\lambda B_1|t|^2
-C_4|t|^{p+1}
+\frac{q+1-\theta}{\theta(q+1)}\lambda|t|^{q+1}\geq0$ for $t\geq0$, since $B_1$
can be chosen arbitrary large.
Thus, it follows from (\ref{eqn:PS4}) that
$\|u_n\|\leq C$ for some $C$ independently of $n$. 

In the second part of the proof we show the strong convergence of $\{u_{n}\}$.
Up to a subsequence,
we suppose that there exist $u\in H_r^1(\R^3)$  such
that
\begin{equation*}
\begin{aligned}
&u_n\rightharpoonup u \quad{\rm weakly \,in }\, H_r^1(\R^3),\\
&u_n\rightarrow u\quad{\rm in}\, L^p(\R^3),\, 2<p<6,\\
&u_n\rightarrow u\quad {\rm a.e.\,in}\, \R^3.
\end{aligned}
\end{equation*}
Hence, passing to the limit in $I_{\lambda}'(u_{n})[u_{n}-u] = o_{n}(1)$
and using that (recall Lemma \ref{Lem:tech})
$$\int_{\mathbb R^{3}} |\phi_{u_{n}}u_{n} ||u_{n} - u|\leq\|\phi_{u_{n}}\|_{6} \|u_{n}\|_{2} \|u_{n}-u\|_{3}
\leq C'\| u_{n}\|^{2}\|u_{n}\|_{2}\|u_{n}-u\|_{3} \to 0$$
and that, by \eqref{f2} (recall that $q<5$)
\begin{eqnarray*}
&&\int_{\mathbb R^{3}} |u_{n}|^{q}|u_{n}-u| \leq \|u_{n}\|_{q+1}^{q} \|u_{n} - u\|_{q+1} \to 0, \\
&&  \|u_{n}\|_{2}\int_{\mathbb R^{3}} |u_{n}||u_{n}-u|\to 0, \\
&& \int_{\mathbb R^{3}} |f(u_{n})||u_{n}-u|\to 0,
\end{eqnarray*}
we conclude that
$$\|u_{n}\|^{2} - \|u\|^{2} =o_{n}(1)$$
and then
$u_n\rightarrow u$ in $H^1_r(\R^3)$.
\end{proof}

It follows from Lemma \ref{Lem:compact}
that for each $\lambda\in(0,1]$, there exists $u_\lambda\in H_r^1(\R^3)$ such that
$$I_\lambda(u_\lambda)=c_\lambda \quad \text{and}\quad I'_\lambda(u_\lambda)=0.$$
 By $I'_\lambda(u_\lambda)[u_\lambda]=0$, we have
\begin{equation}\label{eqn:Pr1}
\int_{\R^3}(|\nabla u_\lambda|^2+u_\lambda^2)
+\int_{\R^3}\phi_{u_\lambda}u_\lambda^2+\lambda\|u_\lambda\|_2^{3}-\int_{\R^3}f(u_\lambda)u_\lambda-
\lambda\int_{\R^3}|u_\lambda|^{q+1}=0.
\end{equation}
Recalling hypothesis \eqref{f3}, it follows from (\ref{eqn:Pr1}) that for $b>0$,
\begin{equation}\label{eqn:Pr2}
b\int_{\R^3}F(u_\lambda)\leq \frac{b}{\varrho}\int_{\R^3}(|\nabla u_\lambda|^2+u_\lambda^2)
+\frac{b}{\varrho}\int_{\R^3}\phi_{u_\lambda}u_\lambda^2+\frac{b\lambda}{\varrho}\|u_\lambda\|_2^{3}
-\frac{b\lambda}{\varrho}\int_{\R^3}|u_\lambda|^{q+1}.
\end{equation}
Moreover, similarly to Lemma \ref{Lem:poh}, we obtain the associated Pohozaev type identity for the modified
problem \eqref{eqn:2-spos}:
\begin{multline}\label{eqn:Pr3}
\frac{1}{2}\int_{\R^3}|\nabla u_\lambda|^2+\frac{3}{2}\int_{\R^3}u_\lambda^2
+2\int_{\R^3}\phi_{u_\lambda}u_\lambda^2 +\frac{3\lambda}{2}\|u_\lambda\|_2^{3} \\
-\frac{3}{2}\int_{\R^3}(1-\sqrt{1-|\nabla \phi_{u_\lambda}|^2})=3\int_{\R^3}F(u_\lambda)+\frac{3\lambda}{q+1}\int_{\R^3}|u_\lambda|^{q+1}.
\end{multline}
Combining (\ref{eqn:Pr2}) with (\ref{eqn:Pr3}), we have for $a\in\R$
\begin{equation}\label{eqn:Pr4}
\aligned
(a+b)\int_{\R^3}F(u_\lambda)\leq& \  \left(\frac{a}{6}+\frac{b}{\varrho}\right)\int_{\R^3}|\nabla u_\lambda|^2
+\left(\frac{a}{2}+\frac{b}{\varrho}\right)\int_{\R^3}|u_\lambda|^2\\
&+\left(\frac{2a}{3}+\frac{b}{\varrho}\right)\int_{\R^3}\phi_{u_\lambda}u_\lambda^2
-\frac{a}{2}\int_{\R^3}\left(1-\sqrt{1-|\nabla u|^2}\right)\\
&+\left(\frac{a}{2}+\frac{b}{\varrho}\right)\lambda\|u_\lambda\|_2^{3}-\left(\frac{a}{q+1}
+\frac{b}{\varrho}\right)\lambda\|u_\lambda\|_{q+1}^{q+1}.
\endaligned
\end{equation}
By letting $a=1-b$ in (\ref{eqn:Pr4}),  using the definition of
$I_\lambda$ and that (recall \eqref{eq:1/2})
$$
\frac{b}{2}\int_{\R^3}(1-\sqrt{1-|\nabla u|^2})\leq \frac{b}{4}\int_{\R^3}\phi_uu^2,
$$
 it follows that
\begin{equation}\label{eqn:Pr5}
\aligned
c_\lambda=I_\lambda(u_\lambda)\geq& \  \left(\frac{1}{3}+\frac{b(\varrho-6)}{6\varrho}\right)\int_{\R^3}|\nabla u_\lambda|^2
+\left(\frac{b}{2}-\frac{b}{\varrho}\right)\int_{\R^3}|u_\lambda|^2\\
&+\left(\frac{1}{2}-\frac{2(1-b)}{3}-\frac{b}{\varrho}\right)\int_{\R^3}\phi_{u_\lambda}u_\lambda^2
-\frac{b}{2}\int_{\R^3}(1-\sqrt{1-|\nabla u_\lambda|^2}) \\
&+\left(\frac{1}{3}-\frac{1-b}{2}-\frac{b}{\varrho}\right)\lambda\|u_\lambda\|_2^{3}
+\left(\frac{1-b}{q+1}+\frac{b}{\varrho}-\frac{1}{q+1}\right)\lambda\|u_\lambda\|_{q+1}^{q+1}\\
\geq& \ \left(\frac{1}{3}+\frac{b(\varrho-6)}{6\varrho}\right)\int_{\R^3}|\nabla u_\lambda|^2
+\left(\frac{b}{2}-\frac{b}{\varrho}\right)\int_{\R^3}|u_\lambda|^2 \\
&+\left(\frac{1}{2}-\frac{2(1-b)}{3}-\frac{b}{\varrho}-\frac{b}{4}\right)\int_{\mathbb R^{3}} \phi_{u_{\lambda}} u_{\lambda}^{2}\\
&+\left(\frac{1}{3}-\frac{1-b}{2}-\frac{b}{\varrho}\right)\lambda\|u_\lambda\|_2^{3}
+\left(\frac{1-b}{q+1}+\frac{b}{\varrho}-\frac{1}{q+1}\right)\lambda\|u_\lambda\|_{q+1}^{q+1}.
\endaligned
\end{equation}
It is easy to check that with the choice $b=2$ all the coefficients above are positive.

\begin{remark}\label{rem:I}
Note that the above computation, from \eqref{eqn:Pr1} to \eqref{eqn:Pr5}, can be repeated for the original functional $I$, that is without the terms having $\lambda$ in front of the integrals. More specifically,
\begin{enumerate}[label=(\alph{*}), ref=\alph{*}]
\item\label{a} if  there exists $u\in H^{1}_{r}(\mathbb R^{3})$ such that $I'(u)=0$, then
 $I(u)\geq C\|u\|^{2}$, with $C$ independent on $u$;
\item\label{b} if $\{u_{n}\}$ is such that $I'(u_{n})=0$ and $\{I(u_{n})\}$ is convergent, then $\{u_{n}\}$
is bounded.
\end{enumerate}

\end{remark}


Then by \eqref{eqn:Pr5}, $c_{\lambda}=I_{\lambda}(u_{\lambda})\geq C \|u_{\lambda}\|^{2}$
 and by Remark \ref{rem:bounded} it follows that
$\{u_\lambda\}_{\lambda\in(0,1]}$ is bounded in $H_r^1(\R^3)$
 uniformly in $\lambda\in(0,1]$; 
 then we can assume,
 as $\lambda\to 0^{+}$,
 $$u_{\lambda}\ \rightharpoonup u_{0} \ \text{ in } \  H^{1}_{r}(\mathbb R^{3})
 \qquad\text{and}\qquad
 I_{\lambda}(u_{\lambda})=c_{\lambda} \to c_{0}>0.$$

 Moreover, for any $v\in H^{1}_{r}(\mathbb R^{3})$, we have
 $$
 I'(u_\lambda)[v]= I'_\lambda(u_\lambda)[v]-\lambda\|u_\lambda\|_2 \int_{\mathbb R^{3}}u_\lambda v
+ \lambda\int_{\R^3} |u_\lambda|^{q}u_\lambda v=o_{\lambda}(1)\|v\|
 $$
 and
 $$I(u_{\lambda}) = I_{\lambda}(u_{\lambda})-\frac{\lambda}{3}\|u_{\lambda}\|_2^{3}
+\frac{\lambda}{q+1}\int_{\R^3}|u_{\lambda}|^{q+1}=c_{0}+o_{\lambda}(1).$$
Thus, $\{u_\lambda\}_{\lambda\in(0,1]}$ is a bounded Palais-Smale sequence for
the  unperturbed functional $I$ at level $c_0$. Arguing similarly as in the second part of the
proof of Lemma \ref{Lem:compact}, by $I'(u_{\lambda})[u_{\lambda} - u_{0}]=o_{\lambda}(1)$,
we get that
$$u_{\lambda} \to u_{0} \ \ \text{on }\ H^{1}_{r}(\mathbb R^{3}), \quad
I'(u_0)=0 \quad\text{and} \quad I(u_0)=c_{0}>0.$$
In particular $u_{0}$ is a nontrivial solution of \eqref{SBI} with $\mu=0$.
Actually we have proved the following fact that will be useful in the future.
\begin{proposition}\label{prop:convergence}
If $\{u_{\lambda}\}_{\lambda\in(0,1]}$ is such that $I'_{\lambda}(u_{\lambda})=0$ and
$c_{\lambda}=I_{\lambda}(u_{\lambda}) \in [m_{1},m_{2}]$, then
there exists $u_{0}\in H^{1}_{r}(\mathbb R^{3})\setminus\{0\}$ such that on a sequence
$\{\lambda_{n}\}$ tending to zero, it holds
$$u_{\lambda_{n}} \to u_{0} \quad \text{in } H^{1}_{r}(\mathbb R^{3}),
\quad c_{\lambda_{n}}\to c_{0}, \quad I(u_{0}) = c_{0} \quad \text{ and } \quad I'(u_{0}) = 0.$$
\end{proposition}

We do not know if $u_{0}$ found above is a ground state solution, however we are now able to give the

\subsection{Proof of Theorem \ref{Thm:EXST}}

Define the set of solutions
$$
\mathcal{S}:=\{u\in H_r^1(\R^3)\setminus\{0\}:\,I'(u)=0\}
$$
that, for what we have proved, is nonempty.
For  $u\in \mathcal{S}$ by \eqref{eq:fF},
 for any $\varepsilon>0$ there exists $C_\varepsilon>0$ such that
$$
\aligned
\|u\|^2\leq \|u\|^2+\int_{\R^3}\frac{|\nabla \phi_u|^2}{\sqrt{1-|\nabla\phi_u|^2}}\leq
\varepsilon\int_{\R^3}u^2+C_\varepsilon\int_{\R^3}|u|^{p+1},
\endaligned
$$
which implies  that $\mathcal S$ is bounded away from zero.


By \eqref{a} in Remark \ref{rem:I},
 there exists some $C>0$ satisfying $I(u)\geq C\|u\|^2$ for all $u\in \mathcal{S}$.
We infer that
$$
c_*:=\inf\limits_{u\in\mathcal{S}}I(u)>0.
$$
Take finally a minimising sequence
$\{u_n\}\subset \mathcal{S}$ so that $I(u_n)\rightarrow c_*$.
By \eqref{b} in Remark \ref{rem:I} we know that
 $\{u_n\}$ is bounded. Similar to the second part of the proof of Lemma \ref{Lem:compact},
there exists $u_{*}\in H_r^1(\R^3)$ so that $u_n\rightarrow u_{*}$ in $H_r^1(\R^3)$ and $I'(u_{*})=0$.
Then $(u_{*}, \phi_{u_{*}})$ is a radial ground state solution of (\ref{SBI}).

\begin{remark}
We observe that if we had perturbed the problem on the left hand
side  with $\lambda \|u\|_{2}^{2\alpha}u$ with $\alpha\in(0,1)$,  everything
would have equally worked.
\end{remark}

\subsection{Proof of Theorem \ref{Thm:multi}}

In this subsection, by using the perturbation approach together with the Symmetric Mountain-Pass
theorem we  prove that equation \eqref{SBI} has infinitely many high energy solutions.

Let $B_{R}$ be the ball of radius $R>0$ of $H^{1}_{r}(\mathbb R^{3})$.
Choose a sequence of finite dimensional subspaces $E_j$ of $H^{1}_r(\R^3)$ with $\dim E_j=j$,
and $R_j>0$ such that $I_\lambda(u)<0$ for $u\in E_j\cap\partial B_{R_j}$.
The existence of such $R_{j}$ is justified by the fact that
in the proof of \eqref{MPii} of Lemma \ref{Lem:MP1} the element $e$ is arbitrary.
Moreover we actually have that $R_{j}$ does not depends on $\lambda$, namely
$$\forall \lambda\in(0,1] : I_{\lambda} (u)<0 \quad\text{ for any } u \in E_j\cap\partial B_{R_j}.$$

Then by  \eqref{MPi} of Lemma \ref{Lem:MP1} the functional $I_{\lambda}$
satisfies all the assumptions of the Symmetric Mountain Pass Theorem
and setting
$$
\Gamma_j=\bigg\{B=\phi(E_j\cap B_{R_j}) | \phi\in C(E_j\cap B_{R_j},H^{1}_r(\R^3)),
\,\phi \,\,\text{is odd},\ \phi=\textrm{Id} \ \text{on} \ E_j\cap \partial B_{R_j}\bigg\},
$$
the minimax values
$$
 c_\lambda(j)=\inf\limits_{B\in\Gamma_j}\sup\limits_{u\in B} I_\lambda(u)\geq\delta>0
 $$
are critical values and  $c_\lambda(j)=I_{\lambda}(u_{\lambda}(j))\rightarrow+\infty$ as
 $j\rightarrow+\infty$, see e.g. \cite{Tanaka}.

For any fixed $j$, by the definition of $ c_\lambda(j)$, we have, recalling Lemma \ref{Lem:tech},
$$
\aligned
 c_\lambda(j)&\leq\sup\limits_{u\in E_j\cap B_{R_j}} I_\lambda(u)\\
 &\leq \sup\limits_{u\in E_j\cap B_{R_j}} \bigg
 \{ C_{1} \|u\|^2+ C_{2}\| u\|^{4} + C_{3}\|u\|^{3}\bigg\}:=\kappa_{R_j},
\endaligned
$$
with $\kappa_{R_j}$ of course  independent of $\lambda\in(0,1]$ and $\|\cdot\|$ is any norm
in $E_{j}$. It follows from
Proposition  \ref{prop:convergence} that there exists $u_{0}(j)\in H^{1}_{r}(\mathbb R^{3})\setminus\{0\}$
such that on a sequence $\lambda_{n}\to 0^{+}$,
$$u_{\lambda_{n}}(j) \to u_{0}(j) \quad \text{in } H^{1}_{r}(\mathbb R^{3}),
\quad c_{\lambda_{n}}(j)\to c_{0}(j)\geq\delta, \quad I(u_{0}(j)) = c_{0}(j) \quad \text{ and } \quad I'(u_{0}(j)) = 0,$$
namely $u_{0}(j)$ is a nontrivial solution of \eqref{SBI}.

If we show that $c_0(j)\rightarrow+\infty$ as $j\to+\infty$,
then  problem \eqref{SBI} has infinitely many solutions and the proof
of Theorem \ref{Thm:multi} is concluded.

Recalling Lemma \ref{Lem:phi}, we estimate $I_\lambda$ as follows
$$
\aligned
I_\lambda(u)=&\frac{1}{2}\int_{\R^3}(|\nabla u|^2+u^2)+\frac{1}{2}\int_{\R^3}\phi_uu^2
-\int_{\R^3}F(u)-\frac{1}{2}\int_{\R^3}(1-\sqrt{1-|\nabla \phi_u|^2})\\
&+\frac{\lambda}{3}\|u\|_2^{3}-\frac{\lambda}{q+1}\int_{\R^3}|u|^{q+1}\\
\geq&\frac{1}{2}\int_{\R^3}(|\nabla u|^2+u^2)-\int_{\R^3}F(u)-\frac{1}{q+1}\int_{\R^3}|u|^{q+1}:=J(u).
\endaligned
$$
Define the  set $\Theta\subset H_r^1(\R^3)$ by
$$
\Theta:=\bigg\{u\in H^{1}_{r}(\mathbb R^{3}): \int_{\R^3}(|\nabla u|^2+u^2)> \int_{\R^3}f(u)u+\int_{\R^3}|u|^{q+1}\bigg\}
\cup\{0\}.
$$
Note that $\partial\Theta$ is the Nehari manifold associated to the even functional $J$,
which, by classical arguments, is bounded away from zero and homeomorphic to the unit sphere.
Then, if $B\in \Gamma_j$ an easy modification of the proof of \cite[Proposition 9.23]{Rabinowitz86}
shows that an intersection property holds so that $\gamma(B\cap\partial\Theta)\geq j$, for all $j\in \mathbb N$.
%
Here $\gamma(\cdot)$
is the Krasnoselski genus of a symmetric set. Hence,
$$
 c_{\lambda}(j)=\inf\limits_{B\in\Gamma_j}\sup\limits_{u\in B}  I_\lambda(u)
 \geq\inf\limits_{A\subset\partial\Theta,\gamma (A)\geq j}\sup\limits_{u\in A} J(u)
:=b(j).
$$
It is not hard to verify that the functional $J$ is bounded below on $\partial\Theta$
and satisfies the Palais-Smale condition. Then  the Ljusternick-Schnirelmann theory
guarantees that  $b(j)$ are diverging critical values for $J$.
 Therefore,
$$
 c_0(j)=\lim\limits_{\lambda\rightarrow0^+}c_{\lambda}(j)\geq b(j)\rightarrow+\infty, \ \ \text{as } j\to+\infty.
$$
That is to say, equation \eqref{SBI} has infinitely many high energy solutions. The proof is complete.

\s{The Critical Case}\label{sec:critical}
\subsection{The modified equation}
In this section we study \eqref{SBI} in the critical case, that is when
$\mu>0$, and without loss of generalities we assume $\mu=1$.

Our aim is to  establish the existence of a ground state solution to \eqref{SBI} with
a general critical nonlinear term.

Motivated by \cite{Liu19}, we  introduce a perturbation technique to overcome this difficulty by modifying system \eqref{SBI}.
The modified problem now is
\begin{equation}\label{eqn:3-spos}
\left\{
 \begin{array}{ll}
-\triangle u+u+\phi u+\lambda\|u\|_2 u=f(u)+|u|^{4}u\,\,&\mbox{in}\,\,\R^3,\medskip\\
-\textrm{div}\displaystyle\bigg(\frac{\nabla\phi}{\sqrt{1-|\nabla\phi|^2}}\bigg)=u^2&\mbox{in}\,\,\R^3,\medskip\\
u(x)\rightarrow0,\,\,\phi(x)\rightarrow0,&\,\text{as}\,\,x\rightarrow\infty,
\end{array}
\right.
\end{equation}
where $\lambda\in(0,1]$. Obviously, its associated energy functional is
$$
J_{\lambda}(u):=I(u)+\frac{\lambda}{3}\|u\|_2^{3}
$$
with $I$ defined in \eqref{eq:I} with $\mu=1$.

Arguing as in Lemma \ref{Lem:MP1}, we see  that
$J_\lambda$  satisfies the Mountain-Pass geometry.
\begin{lemma}\label{Lem:MP2}
Suppose that \eqref{f1}-\eqref{f3}  hold.
 Then 
 \begin{enumerate}[label=(\roman{*}), ref=\roman{*}]
 \item\label{MPi} there exist $\rho,\delta>0$ such that, for any $\lambda\in(0,1]$,
  $J_{\lambda}(u)\geq \delta$ for every $u\in S_\rho=\{u\in E: \|u\|=\rho\}$;
 \medskip
 \item\label{MPii} there is $v\in H_r^1(\R^3)$ with $\|v\|>\rho$ such that,
 for any $\lambda\in(0,1]$, $J_{\lambda}(v)<0$.
 \end{enumerate}
\end{lemma}

Then  there exists a $(PS)_{c_{\lambda}}$ sequence
$\{u_n\}\subset H_r^1(\R^3)$, that is,
\begin{equation*}
J_{\lambda}(u_n)\rightarrow c_{\lambda}\quad\text{and}\quad J'_{\lambda}(u_n)\rightarrow0,
\end{equation*}
where $c_{\lambda}$ is the Mountain Pass level characterized by
\begin{equation*}
c_{\lambda}=\inf\limits_{\gamma\in\Gamma}\max\limits_{t\in[0,1]}J_{\lambda}(\gamma(t))
\end{equation*}
with
$$
\Gamma_{\lambda}:=\{\gamma\in C^1([0,1],H_r^1(\R^3)):\,\gamma(0)=0\quad\text{and}\quad J_{\lambda}(\gamma(1))<0\}.
$$
\begin{remark}
As in Remark \ref{rem:bounded}, we see that $\{c_{\lambda}\}_{\lambda\in(0,1]}$ is bounded away from zero and bounded above.
\end{remark}

 In the
following, we will give an upper bound  for $c_\lambda$ which will be of use in proving
the convergence of Palais-Smale sequences.
In what follows
$$S=\inf_{v\in D^{1,2}(\mathbb R^{3})\setminus\{0\}} \frac{\displaystyle \int_{\R^3}|\nabla v|^2}
{\Big(\displaystyle\int_{\R^3}|v|^6\Big)^{1/3}}.$$

\begin{lemma}\label{Lem:estimate}
Assume that \eqref{f1}-\eqref{f4}
hold. If $r\in(4,6)$, or $r\in(2,4]$ and $D$ is sufficiently large, then $c_\lambda<\frac{1}{3}S^{{3}/{2}}$.
\end{lemma}
\begin{proof}
Since for $\lambda\in(0,1)$ it is  $J_{\lambda} \leq J_{1}$ and $\Gamma_{1} \subset \Gamma _{\lambda}$, it is sufficient to prove the result with $\lambda=1$. We will show that there exists
$\gamma\in \Gamma_{1}$ such that
$$c_{1} \leq \max_{t\in[0,1]} J_{1}(\gamma(t))<\frac13 S^{3/2}.$$

  Let $\phi\in C_0^{\infty}(\R^3)$ be a radial cut-off function with support
in $B_{2a}(0)$ so that $0\leq \phi(x)\leq 1$ and $\phi(x)\equiv1$ on $B_a(0)$, where $a>0$.

It is well-known that $S$ is attained on the functions $\frac{\epsilon^{1/4}}{(\epsilon+|x|^2)^{1/2}}$ for $\epsilon>0$.
Defining $U_\epsilon(x)=\phi(x)\frac{\epsilon^{1/4}}{(\epsilon+|x|^2)^{1/2}}$,
 a direct calculation gives, as $\varepsilon\to 0$,
\begin{equation}\label{eqn:3iu1}
\int_{\R^3}|\nabla U_\epsilon|^2=S^{3/2}+O(\epsilon^{{1}/{2}}),\\
\quad \int_{\R^3}|U_\epsilon|^6=S^{3/2}+O(\epsilon^{{3}/{2}})
\end{equation}
and
\begin{equation}\label{eqn:3iu2}
\int_{\R^3}|U_\epsilon|^t=\left\{
  \begin{array}{ll}
    O(\epsilon^{{t}/{4}}),\quad&t\in[2,3); \smallskip \\
    O(\epsilon^{{3}/{4}}|\ln\epsilon|),\quad& t=3;\smallskip \\
    O(\epsilon^{{(6-t)}/{4}}),\quad& t\in(3,6).
 \end{array}
\right.
\end{equation}
From (\ref{eqn:3iu1}) we have
\begin{equation}\label{eqn:3iu3}
\frac{\displaystyle \int_{\R^3}|\nabla U_\epsilon|^2}
{\Big(\displaystyle\int_{\R^3}|U_\epsilon|^6\Big)^{1/3}}=S+O(\epsilon^{{1}/{2}}).
\end{equation}


Define the function $$y_{\varepsilon}(t):=\frac{t^2}{2}\|U_\epsilon\|^2-\frac{ t^6}{6}\int_{\R^3}|U_\epsilon|^6.$$
It is easy to check that $y_{\varepsilon}$ attains its maximum at
\begin{equation}\label{eq:Te}
T_{\varepsilon}=\left(\frac{\|U_\epsilon\|^2}{\displaystyle\int_{\R^3}|U_\epsilon|^6}\right)^{1/4}
\quad \text{ and }  \quad
y_{\varepsilon}(T_{\varepsilon})=\frac{1}{3}\frac{\| U_\epsilon\|^3}{\|U_\epsilon\|_6^3} = \frac{1}{3}S^{3/2} + O(\varepsilon^{1/2}).
\end{equation}

\medskip

{\bf Claim. }
 For $\varepsilon$ small enough,
$\max_{t\geq0} J_{1}(t  U_{\varepsilon}) < \frac13 S^{3/2}.$
\medskip

We show that the interval $[0,+\infty)$ can be divided into three subintervals
on which the inequality of the Claim holds.

Since the $H^{1}$ and $L^{p}$ norms of $\{U_{\varepsilon}\}$
are bounded as $\varepsilon$ tends to zero, it is possible
to take $t'\in(0,1)$ such that for any $\epsilon\in(0,1)$, we have
\begin{equation}\label{eqn:3iu4}
\aligned
\max\limits_{t\in[0,t']}J_1(t {U_\epsilon})&\leq
\max\limits_{t\in[0,t']}\left(\frac{t^2  }{2}\| U_\epsilon\|^2
+\frac{t^2 }{2}\int_{\R^3}\phi_{t U_\epsilon} (U_\epsilon)^2+\frac{ t^{3}}
{3}\| U_{\varepsilon}\|_2^{3}\right)\\
&\leq\max\limits_{t\in[0,t']}\left(\frac{t^2   }{2}\|U_\epsilon\|^2
+C'\frac{t^4    }{2}\| U_\epsilon\|_{{12}/{5}}^4
+\frac{ t^{3} }{3}\| U_{\varepsilon}\|_2^{3}\right)\\
&<\frac{1}{3}S^{3/2}.
\endaligned
\end{equation}
Moreover, using hypothesis \eqref{f4}, one has
\begin{equation}\label{eqn:3iu5}
J_1(t U_\epsilon)\leq y_{\varepsilon} (t )+C'\frac{t^4}{2}\| U_\epsilon\|_{{12}/{5}}^4+\frac{t^{3}}{3}\|
U_{\varepsilon}\|_2^{3}-\frac{ D}{r}t^r\int_{\R^3}| U_\epsilon|^r
:=e_{\varepsilon}(t).
\end{equation}
It follows from Lemma \ref{Lem:MP2} and (\ref{eqn:3iu5}) that, for any $\varepsilon>0$,
$\lim\limits_{t\rightarrow+\infty}J_1(t {U_\epsilon})=-\infty$,
 $J_1(t {U_\epsilon})>0$ as $t$ is close to $0$; moreover
 there exists $t_\epsilon>0$ such that $e_{\varepsilon}(t_\epsilon)=0$
 and $e_{\varepsilon}(t)<0$ for $t>t_\epsilon$.
In particular,
\begin{equation}\label{eq:bo}
J_{1}((t_{\varepsilon}+1)U_{\varepsilon})\leq e_{\varepsilon}(t_{\varepsilon}+1)<0.
\end{equation}
Note that, as $\varepsilon$ tends to zero, $t_{\varepsilon}$ is bounded. Indeed, from
\begin{eqnarray*}\label{eq:}
\aligned
0&=e_{\varepsilon}(t_\epsilon)\\
&=t_\epsilon^2\left(\frac{1}{2}\| U_\epsilon\|^2+C'\frac{t_{\varepsilon}^2}{2}\| U_\epsilon\|
_{{12}/{5}}^4+\frac{t_{\varepsilon}}{3}\|  U_{\varepsilon}\|_2^{3}
-\frac{ t_\epsilon^4}{6}\int_{\R^3}| U_\epsilon|^6-\frac{ D}{r}t_{\varepsilon}^{r-2}\int_{\R^3}| U_\epsilon|^r\right),
\endaligned
\end{eqnarray*}
we deduce
\begin{equation*}
\aligned
\frac{1}{2}\| U_\epsilon\|^2+C'\frac{t_\epsilon^2}{2}\| U_\epsilon\|_{{12}/{5}}^4
+\frac{t_\epsilon}{3}\| U_{\varepsilon}\|_2^{3}
&=\frac{t_\epsilon^4}{6}\int_{\R^3}| U_\epsilon|^6
+\frac{D}{r}t_{\varepsilon}^{r-2}\int_{\R^3}| U_\epsilon|^r \geq\frac{t_\epsilon^4}{6}\int_{\R^3}|
U_\epsilon|^6,
\endaligned
\end{equation*}
%
%
and then, since the $H^{1}$ and $L^{p}$ norms of $U_{\varepsilon}$ are bounded,
$ t_{\varepsilon}$ has to be bounded as $\varepsilon$ varies in, let us say, $(0,\varepsilon_{0})$.
We deduce then that
$\lim_{t\rightarrow+\infty}e_{\varepsilon}(t)<0$,   uniformly in
$\epsilon\in(0,\epsilon_0)$.
As a consequence, by \eqref{eqn:3iu5},
\begin{equation*}
\lim_{t\rightarrow+\infty}J_1(t U_\epsilon)<0  \quad \text{uniformly in } \
\epsilon\in(0,\epsilon_0).
\end{equation*}
Thus there exists $t''>t^*$ such that for any $\epsilon\in(0,\epsilon_0)$,
\begin{equation}\label{eq:2s}
\max\limits_{ t\geq t''}J_1(t {U_\epsilon})<\frac{1}{3}S^{3/2}.
\end{equation}


\medskip
Finally, from  (\ref{eqn:3iu2}), (\ref{eqn:3iu3}), \eqref{eq:Te}  and (\ref{eqn:3iu5}), we infer that
\begin{equation}\label{eqn:3iu7}
\aligned
\max\limits_{t\in[t', t''] }J_1(t {U_\epsilon})
&\leq y_{\varepsilon}(t)+C'\frac{t^4}{2}\| U_\epsilon\|_{{12}/{5}}^4+\frac{ t^{3}}{3}\|U_{\varepsilon}\|
_2^{3}-\frac{D}{r}\int_{\R^3}|
U_\epsilon|^{r}\\
&\leq\frac{1}{3}S^{3/2}+O(\epsilon^{{1}/{2}})-\frac{D}{r}\int_{\R^3}| U_\epsilon|^{r}.
\endaligned
\end{equation}
From this it follows that, uniformly in $\varepsilon\in (0,\varepsilon_{0})$,
\begin{equation}\label{eqn:3iu8}
\max\limits_{t\in[t', t'']}J_1(t  {U_\epsilon})<\frac{1}{3}S^{3/2}.
\end{equation}
Indeed if $r\in(2,4]$ and $D$
sufficiently large, $\epsilon\in(0,\epsilon_0)$ fixed, the conclusion follows
from (\ref{eqn:3iu2}) and (\ref{eqn:3iu7}).
On the other hand, if  $r\in(4,6)$, since $\frac{6-r}{4}<\frac{1}{2}$,
\eqref{eqn:3iu8} follows from (\ref{eqn:3iu2}) possibly reducing $\varepsilon$.

Then the Claim follows by \eqref{eqn:3iu4}, \eqref{eq:2s} and \eqref{eqn:3iu8}
since $t'$ and $t''$ do not depend on $\varepsilon$ (whenever it small).

\medskip


Define the curve
$$\xi_{\varepsilon}: t\in[0,+\infty) \mapsto t (t_{\varepsilon}+1)U_{\varepsilon} \in H^{1}_{r}(\mathbb R^{3})$$
and note that, since $\xi_{\varepsilon}(1) = (t_{\varepsilon}+1)U_{\varepsilon}$, by \eqref{eq:bo}
and the Claim it holds
$$\gamma_{\varepsilon}:=\xi_{\varepsilon} |_{[0,1]} \in \Gamma_{1} \quad \text{ and } \quad  c_{1} \leq \max_{t\in [0,1]}
J_{1}(\gamma_{\varepsilon}(t))
\leq \max_{t\geq0} J_{1}(\xi_{\varepsilon}(t)) < \frac13 S^{3/2}$$
and the proof is completed.
\end{proof}

As a consequence we get the Palais-Smale condition
below the level $\frac13 S^{3/2}$.

\begin{lemma}\label{Lem:2-compact}
For fixed $\lambda\in (0,1]$, let $\{u_n\}\subset H_r^1(\R^3)$ be a $(PS)_{c_{\lambda}}$ sequence of $J_{\lambda}$.
Then there exists $u\in H_r^1(\R^3)$ such that
$u_n\rightarrow u$ in $H_r^1(\R^3)$.
\end{lemma}
\begin{proof}
We first show that the sequence $\{u_n\}$ is bounded in $H_r^1(\R^3)$.
Taking $\theta\in(4,6)$, by Lemma \ref{Lem:tech} there exist
$C_i>0,i=1,\ldots,4$ such that
\begin{equation*}
\aligned
C_1+C_2\|u_n\|&\geq J_{\lambda}(u_n)-\frac{1}{\theta}J'_\lambda(u_n)[u_n]\\
&\geq \frac{\theta-2}{2\theta}\|u_n\|^2-\frac{\theta-2}{2\theta}\int_{\R^3}\phi_{u_n}u_n^2
+\frac{\theta-3}{3\theta}\lambda\|u_n\|_2^{3}\\
&+\int_{\R^3}(\frac{1}{\theta}f(u_n)u_n-F(u_n))-\frac{1}{2}\int_{\R^3}(1-\sqrt{1-|\nabla \phi_{u_n}|^2})
+\frac{6-\theta}{6\theta}\int_{\R^3}|u_n|^{6}\\
 &\geq \frac{\theta-2}{2\theta}C_3\|u_n\|^2-\frac{\theta-2}{2\theta}\int_{\R^3}\phi_{u_n}u_n^2
+\frac{\theta-3}{3\theta}\lambda\|u_n\|_2^{3}\\
&-C_4\int_{\R^3}|u_n|^{p+1}-\frac{1}{2}\int_{\R^3}(1-\sqrt{1-|\nabla \phi_{u_n}|^2})
+\frac{6-\theta}{6\theta}\int_{\R^3}|u_n|^{6},
\endaligned
\end{equation*}
which implies by (\ref{eqn:PS2}) and Lemma \ref{Lem:phi} that
\begin{equation*}
\aligned
C_1+C_2\|u_n\|
\geq &\frac{\theta-2}{2\theta}C_3\|u_n\|^2+\frac{\theta-4}{4\theta}\int_{\R^3}\phi_{u_n}u_n^2
+\frac{\theta-3}{3\theta}\|u_n\|_2^{3}\\
&-C_4\int_{\R^3}|u_n|^{p+1}
+\frac{6-\theta}{6\theta}\int_{\R^3}|u_n|^{q+1}\\
\geq &\frac{\theta-2}{2\theta}C_3\|u_n\|^2
+\frac{\theta-3}{3\theta}\|u_n\|_2^{3}\\
&-C_4\int_{\R^3}|u_n|^{p+1}
+\frac{6-\theta}{6\theta}\int_{\R^3}|u_n|^{q+1}.
\endaligned
\end{equation*}
Then arguing similarly as in Lemma \ref{Lem:compact}, we deduce that
$\|u_n\|\leq C$ for some $C$ independently of $n$.
Thus, there exists a subsequence of $\{u_n\}$ (still denoted by $\{u_n\}$, without loss of generality)
 such
that
\begin{equation}\label{eqn:4-imbedding}
\begin{aligned}
&u_n\rightharpoonup u \quad{\rm weakly \,in }\, H_r^1(\R^3),\\
&u_n\rightarrow u\quad{\rm in}\, L^p(\R^3),\, 2<p<6,\\
&u_n\rightarrow u\quad {\rm a.e.\,in}\, \R^3
\end{aligned}
\end{equation}
for some $u\in H_r^1(\R^3)$.
Moreover, we assume that there exists $A\geq0$ such that
$\|u_n\|_2\rightarrow A$ as $n\rightarrow\infty$. Define the energy functional
\begin{eqnarray*}
J_{A,\lambda}(v)&:=&J_{\lambda}(v) -\frac{\lambda}{3}\|v\|_2^{3} +\frac{\lambda A}{2}\|v\|_{2}^{2}\\
&=&\frac{1}{2}\|v\|^2+\frac{\lambda A}{2}\|v\|_2^2+\frac{1}{2}\int_{\R^3}\phi_{v}v^2
-\frac{1}{2}\int_{\R^3}(1-\sqrt{1-|\nabla \phi_{v}|^2})-\int_{\R^3}F(v)-\frac{1}{6}\int_{\R^3}|v|^6.
\end{eqnarray*}
Since  $\{u_n\}$ is a bounded Palais-Smale sequence for $J_\lambda$ at level $c_{\lambda}$, we easily get
$$J_{A,\lambda}(u_n)\rightarrow c_\lambda+\frac{\lambda A^3}{6} \quad\text{ and } \quad
J'_{A,\lambda}(u_n)\rightarrow0  \ \text{ in  $H^{-1}$ as }n\rightarrow\infty.$$
 Here $H^{-1}$ is the dual space of $H_r^1(\R^3)$. Standard
 computations give that $J'_{A,\lambda}(u)=0$.
Observe from Lemma \ref{Lem:poh} that we can get the associated Pohozaev type identity
to the functional  $J_{A,\lambda}$ which is
\begin{multline}\label{eqn:4-Pr3}
\frac{1}{2}\int_{\R^3}|\nabla u|^2+\frac{3}{2}\int_{\R^3}u^2
+2\int_{\R^3}\phi_{u}u^2 +\frac{3\lambda}{2}A\|u\|_2^{2}\\
-\frac{3}{2}\int_{\R^3}(1-\sqrt{1-|\nabla \phi_{u}|^2})
=3\int_{\R^3}F(u)+\frac{3}{6}\int_{\R^3}|u|^{6}.
\end{multline}
From \eqref{f3} and $J'_{A,\lambda}(u)[u]=0$, we have
\begin{equation}\label{eqn:4-Pr4}
2\int_{\R^3}F(u)\leq \frac{2}{\varrho}\int_{\R^3}(|\nabla u|^2+u^2)
+\frac{2}{\varrho}\int_{\R^3}\phi_{u}u^2+\frac{2\lambda A}{\varrho}\|u\|_2^{2}
-\frac{2}{\varrho}\int_{\R^3}|u|^{6}.
\end{equation}
It follows from (\ref{eqn:4-Pr3}) and \eqref{eqn:4-Pr4} that
\begin{multline*}\label{eqn:4-Pr4+}
\int_{\R^3}F(u)\leq \left(\frac{2}{\varrho}-\frac{1}{6} \right)\int_{\R^3}|\nabla u|^2
+\left(\frac{2}{\varrho}-\frac{1}{2}\right )\int_{\R^3}u^2+\left(\frac{2}{\varrho}-\frac{2}{3}\right)\int_{\R^3}\phi_{u}u^2
\\
+\left(\frac{2}{\varrho}-\frac{1}{2}\right)\lambda A\|u\|_2^{2}
+\left(\frac{1}{6} -\frac{2}{\varrho} \right)\int_{\R^3}|u|^{6}+
\frac{1}{2}\int_{\R^3}(1-\sqrt{1-|\nabla\phi_u|^2}).
\end{multline*}
Using the definition of
$J_{A,\lambda}$,
 we have
\begin{equation}\label{eqn:4-Pr5}
\aligned
J_{A,\lambda}(u)\geq& \  \left(\frac{2}{3}-\frac{2}{\varrho}\right)\int_{\R^3}|\nabla u|^2
+\left(1-\frac{2}{\varrho}\right)\int_{\R^3}u^2
+\left(\frac{7}{6}-\frac{2}{\varrho}\right)\int_{\R^3}\phi_{u}u^2\\
&-\int_{\R^3}(1-\sqrt{1-|\nabla \phi_u|^2})
+\left(1-\frac{2}{\varrho}\right)\lambda A\|u\|_2^{2}
+\left(\frac{2}{\varrho}-\frac{1}{3}\right)\|u\|_{6}^{6}\\
\geq& \  \frac{1}{3}\lambda A\|u\|_2^{2}.
\endaligned
\end{equation}
Moreover,  defining $w_n:=u_n-u$,
 as in the proof of Lemma \ref{Lem:compact}, from $J'_{A,\lambda}(u_{n})[w_{n}]=o_{n}(1)$ we get,
\begin{equation*}
\|u_{n}\|^{2} - \|u\|^{2}+\lambda A\|u_{n}\|_2^{2} -
\lambda A\|u\|_2^{2}-\|u_{n}\|^{6}_{6}+\|u\|_{6}^{6} = o_{n}(1).
\end{equation*}
and by  means of the Brezis-Lieb lemma (see \cite{Brezis83})
\begin{equation}\label{eqn:eq:}
\|w_{n}\|^{2} +\lambda A\|w_{n}\|_2^{2}-\|w_{n}\|^{6}_{6} = o_{n}(1).
\end{equation}

As in Lemma \ref{Lem:phi} let
$$
E_u(\phi_{u})=\int_{\R^3}(1-\sqrt{1-|\nabla\phi_{u}|^2})-\int_{\R^3}\phi_{u} u^2.
$$
{\bf Claim: } $E_{u_{n}}(\phi_{u_{n}}) - E_{u}(\phi_{u})=o_{n}(1)$.
From \cite[Remark 5.5]{BdP},  the convergence in  \eqref{eqn:4-imbedding} implies that
 $\phi_{u_n}$ converges to $\phi_u$ weakly in $X$ and uniformly in
$\R^3$. Since $X$ is weakly closed, we get $\phi_u\in X$.
On the other hand, recalling Lemma \ref{Lem:compact0}, the following holds:
\begin{eqnarray*}
&& \int| \phi_{u}u_{n}^{2} -\phi_{u}u^{2}| \leq C \|\nabla \phi_{u}\|_{2}\|u_{n}^{2} - u^{2}\|_{6/5}
 \to 0
 \end{eqnarray*}
and so it is easy to see that
\begin{equation}\label{eqn:4-PS4+}
E_{u_n}(\phi_{u})\rightarrow E_{u}(\phi_{u}),\quad\text{as}\,\,n\rightarrow\infty.
\end{equation}
In view of Lemma \ref{Lem:phi}, one has $E_{u_n}(\phi_{u_n})\leq E_{u_n}(\phi_{u})$
that joint with \eqref{eqn:4-PS4+}, gives that
\begin{equation}\label{eqn:4-PS5}
\limsup\limits_{n\rightarrow\infty}E_{u_n}(\phi_{u_n})\leq E_{u}(\phi_{u}).
\end{equation}
As in \cite[Lemma 2.4]{Yu2010}, one can conclude that
\begin{equation}\label{eqn:4-PS6}
E_u(\phi_u)\leq\liminf\limits_{n\rightarrow\infty}E_{u}(\phi_{u_n}).
\end{equation}
Recalling once again Lemma \ref{Lem:compact0}, we have
$$
 \int|\phi_{u_{n}}u_{n}^{2} -\phi_{u_n}u^{2}|\leq C \|\nabla \phi_{u_{n}}\|_{2}\|u_{n}^{2} - u^{2}\|_{6/5}=o_n(1),
$$
which implies  that $E_{u}(\phi_{u_n})=E_{u_n}(\phi_{u_n})+o_n(1)$. Hence, from (\ref{eqn:4-PS6}) we have
\begin{equation}\label{eqn:4-PS6+}
E_u(\phi_u)\leq\liminf\limits_{n\rightarrow\infty}E_{u_n}(\phi_{u_n}).
\end{equation}
The Claim follows by \eqref{eqn:4-PS5} and \eqref{eqn:4-PS6+}.

\medskip

By the Brezis-Lieb Lemma and  using the Claim we arrive at
\begin{eqnarray*}
J_{A,\lambda}(u_n)-J_{A,\lambda}(u)&=&
\frac{1}{2}\|w_n\|^2+\frac{\lambda A}{2}\|w_n\|_2^2-E_{u_n}(\phi_{u_n})+E_{u}(\phi_{u})-\frac{1}{6}\|
w_n\|^6_{6}\\
&=&
\frac{1}{2}\|w_n\|^2+\frac{\lambda A}{2}\|w_n\|_2^2
-\frac{1}{6}\|w_n\|_6^6+o_{n}(1)
\end{eqnarray*}
and by (\ref{eqn:eq:}) we have
\begin{equation*}
J_{A,\lambda}(u_n)-J_{A,\lambda}(u)=\frac{1}{3}\|w_n\|^2+\frac{\lambda A}{3}\|w_n\|_2^2+o_{n}(1).
\end{equation*}
Then, taking into account  Lemma \ref{Lem:estimate} and (\ref{eqn:4-Pr5}) one has
\begin{eqnarray*}
\frac{1}{3}S^{3/2}+\frac{\lambda A^3}{6}-\frac{\lambda A}{3}\|u\|_2^{2}
&>& c_\lambda+\frac{\lambda A^3}{6}-J_{A,\lambda}(u)\\
&=& J_{A,\lambda}(u_n)-J_{A,\lambda}(u)+o_{n}(1)\\
&=&\frac{1}{3}\|w_n\|^2+\frac{\lambda A}{3}\|w_n\|_2^2 +o_{n}(1)\\
&=&\frac{1}{3}\|w_n\|^2 +\frac{\lambda A}{3}\|u_{n}\|_{2}^{2} - \frac{\lambda A}{3}\|u\|_{2}^{2}+o_{n}(1)
\end{eqnarray*}
from which it follows
$\frac{1}{3}S^{3/2}>\frac{1}{3}\|w_n\|^2+o_{n}(1)$,
and then
\begin{equation}\label{eqn:4-PS7}
\|w_{n}\|^{4/3}  +o_{n}(1) < S.
\end{equation}
On the other hand, from \eqref{eqn:eq:} we get $\|w_{n}\|^{6}_{6} \geq \|w_{n}\|^{2}+o_{n}(1)$
and then  $\|w_{n}\|_{6}^{2} \geq \|w_{n}\|^{2/3} +o_{n}(1).$
As a consequence, recalling the definition of $S$,
$$\|w_n\|^2\geq S\|w_n\|_6^2 \geq S\|w_{n}\|^{2/3}+o_{n}(1)$$
and then, if $\|w_{n}\|\not\to0$,
$$S\leq \|w_{n}\|^{4/3} +o_{n}(1)$$
which is in contradiction with \eqref{eqn:4-PS7} and concludes the proof.
\end{proof}

We give now the
\subsection{Proof of Theorem \ref{th:crit}}
Indeed now the proof follows as for Theorem \ref{Thm:EXST} so we simply sketch  it.

By Lemma \ref{Lem:2-compact} we deduce
that for fixed $\lambda\in(0,1]$, there exists $u_\lambda\in H_r^1(\R^3)$ such that
$$J_\lambda(u_\lambda)=c_\lambda \quad\text{ and} \quad J'_\lambda(u_\lambda)=0$$
which  gives a solution of \eqref{eqn:3-spos}.
Note now that the computations from \eqref{eqn:Pr1} to \eqref{eqn:Pr5}
can be repeated  even for the functional $J_{\lambda}$ and then
we obtain  that $c_{\lambda}=J_{\lambda}(u_{\lambda} ) \geq C\|u_{\lambda}\|^{2}$.
Arguing as before, as $\lambda\to 0^{+}$,
 $$u_{\lambda}\ \rightharpoonup u_{0} \ \text{ in } \  H^{1}_{r}(\mathbb R^{3})
 \qquad\text{and}\qquad
 J_{\lambda}(u_{\lambda})=c_{\lambda} \to c_{0}>0.$$
 For any $\varphi\in C_0^\infty(\R^3)$, we have
 $$
 J'(u_{\lambda})[\varphi]= J'_{\lambda}(u_{\lambda})[\varphi]-\lambda\|u_{\lambda}\|_2  \int_{\mathbb R^{3}} u_\lambda\varphi=o_{\lambda}(1)\|\varphi\|
 $$
 and
 $$J(u_{\lambda}) = J_{\lambda}(u_{\lambda}) -\frac{\lambda}{3}\|u_{\lambda}\|_{2}^{3}=c_{0}+o_{\lambda}(1)$$
Thus, $\{u_\lambda\}_{\lambda\in(0,1]}$ is a Palais-Smale sequence for $J$ with level $c_0$.
Then $u_{0}$ is a nontrivial solution of \eqref{SBI} with $\mu=1$.
Observing that Remark \ref{rem:I} is valid also for the functional $J$, the proof is concluded
exactly as before.

\end{document}